\documentclass[12pt]{amsart}%
\usepackage{amsmath}
\usepackage{amsfonts}
\usepackage{amssymb}
\usepackage{color}
\usepackage{graphicx}
\usepackage{amsthm}
\usepackage{mathrsfs}
\usepackage{fancyhdr}%
\newcommand{\wiener}{\mathcal{W}}
\newcommand{\rkhs}{\mathcal{H}}
\newcommand{\prob}{\mathbb{P}}
\newcommand{\reals}{\mathbb{R}}
\newcommand{\E}{\mathbb{E}}

\newcommand{\half}{\frac{1}{2}}

\newcommand{\tB}{\tilde{B}}

\newcommand{\tmu}{\tilde{\mu}}
\newcommand{\tX}{\widetilde{X}}
\newcommand{\borel}{\mathcal{B}}
\newcommand{\norm}[1]{\left\|{#1}\right\|}

\newtheorem{theorem}{Theorem}[section]
\theoremstyle{plain}

\newtheorem{case}{Case}

\newtheorem{corollary}[theorem]{Corollary}

\newtheorem{lemma}[theorem]{Lemma}

\newtheorem{proposition}[theorem]{Proposition}
\newtheorem{remark}[theorem]{Remark}

\numberwithin{equation}{section}
\begin{document}
\title[Smoothness of Density for Area of fBm]{Smoothness of Density for the Area Process of Fractional Brownian Motion}
\author{Patrick Driscoll}

\begin{abstract}
We consider a process given by a two-dimensional fractional Brownian
motion with Hurst parameter $\frac{1}{3}<H<\frac{1}{2}$, along with an
associated L\'{e}vy area, and prove the smoothness of a density for this
process with respect to Lebesgue measure.

\end{abstract}
\maketitle


\section{Introduction}

Let $B_{t} := (B^{1}_{t},B^{2}_{t}),\,t\in[0,T]$ be two-dimensional fractional
Brownian motion of Hurst parameter $\frac{1}{3} < H < \frac{1}{2}$,
and let 
\[
(B_{m})_{t} := \left(  (B^{1}_{m})_{t}, (B^{2}_{m})_{t}\right)
\]
denote the $m$-th dyadic approximation of $B$ (as defined below in Section~\ref{dyadicapproximation}). Define the area processes
\begin{equation}\label{e.am}
(A_{m})_{t} := \frac{1}{2} \left[  \int_{0}^{t} (B^{1}_{s})_{m} d(B^{2}%
_{s})_{m} - \int_{0}^{t} (B^{2}_{s})_{m} d(B^{1}_{s})_{m}\right]  ,
\end{equation}
and $A_{t} := \lim\limits_{m\rightarrow\infty} (A_{m})_{t}$ (where this
convergence is almost sure - see Theorem 2 of \cite{coutin_qian}).

The main result of this paper is as follows:

\begin{theorem}
\label{mainresult}Define the random process $\{Y\}_{0\leq t \leq T}$, taking
values in $\mathbb{R}^{3}$, by
\begin{align*}
Y_{0}  &  =0,\\
Y_{t}  &  =(B_{t},A_{t}).\qquad(t\in(0,T])
\end{align*}
Then for all $t \in [0,T]$, the density of $Y_t$ with respect to Lebesgue measure is $C^{\infty}$.

\end{theorem}

The investigation of this process is motivated by the potential for fractional
Brownian motion to be a useful driving signal in stochastic differential
equations that model a wide variety of natural and financial phenomena; in
particular, the presence of long-range persistence (for $H > 1/2$) or
anti-persistence (for $H < 1/2$) makes fBm a natural candidate for a driving
process in many scenarios. Several examples of such applications are included
in \cite{shiryaev} and \cite{norros}.

One area of interest in the study of stochastic differential
equations is on finding sufficient conditions for existence and regularity of
densities for solutions. More specifically, given some solution $\{Y_{t}\}$ to
the equation
\begin{equation}\label{stochasticequation}
dY_{t} = \sum\limits_{i=1}^{d} X_{i}(Y_{t})d\xi_{t},
\end{equation}
where $\{X_{i}\}$ is some collection of vector fields and $\xi_{t}$ is a
Gaussian driving process on the space $\mathcal{C}([0,T],\mathbb{R}^{d})$, it
is natural ask whether $\xi$ admits a density with respect to Lebesgue
measure. It is due to the celebrated theorem of H\"ormander (see, for example,
Theorem 38.16 of \cite{rogerswilliams2}) that, when our driving process is
standard Brownian motion, our solution admits a smooth density so long as the
set of vectors $\{X_{i}, [X_{i},X_{j}], [[X_{i},X_{j}],X_{k}]\ldots\}$ spans
$\mathbb{R}^{d}$.  

In the case of fractional Brownian motion, one may no longer appeal to the types of martingale arguments used in proofs of the above result in the standard case.  When $H > \half$, the positive correlation of increments of sample paths results in better variational properties than those of Brownian motion, and so one may use Young's integration theory to attack the problem -- existence of a density to a solution of \eqref{stochasticequation} under this condition is proven in \cite{nualart_saussereau}, and smoothness is proven in \cite{baudoin_hairer}.  When $H < \half$, one must turn to the rough path theory of T. Lyons (see \cite{lyons}) in order to interpret \eqref{stochasticequation} in a meaningful manner.  The connection between fractional Brownian motion with $\frac{1}{4} < H < \half$ and rough paths is investigated in \cite{coutin_qian}, and existence of a density in the case of $\frac{1}{3} < H \half$ is proven in \cite{cass_friz_victoir}.  As far as we are aware, Theorem~\ref{mainresult} is the first positive result involving smoothness, and may give hope that similar results will hold in a more general setting.

In order to prove Theorem~\ref{mainresult}, we appeal to the usual technique of Malliavin calculus.  It follows from
Theorem 5.1 of \cite{malliavin} that it is enough to show that $Y$ satisfies the following two conditions

\begin{enumerate}
\item[1.] $Y \in\mathbb{D}^{\infty}$;

\item[2.] If $\gamma=DY(DY)^{\ast}$ is the Malliavin covariance matrix
associated to $Y$, then
\[
(\det\gamma)^{-1}\in L^{\infty-}(\mathcal{W}^{2},\mathbb{P}):=\bigcap
\limits_{j\geq1}L^{j}(\mathcal{W}^{2},\mathbb{P}).
\]

\end{enumerate}

We will begin by calculating the derivative $DY$ explicitly, and from this Condition 1 will be proven in Proposition~\ref{p.smooth}.  Condition 2 will then follow by direct analysis of the Malliavin covariance matrix; see Proposition~\ref{p.mal} and Corollary~\ref{c.int}.
\section{Background}

\subsection{Fractional Brownian Motion}

A (one-dimensional) fractional Brownian motion $\{B^{H}_{t};\; (t \in[0,T])\}$
of Hurst parameter $H \in[0,1]$ is a continuous-time centered Gaussian process
with covariance given by
\[
\mathbb{E}(B^{H}_{s}B^{H}_{t}) = R(s,t) := \frac{1}{2}\left(  s^{2H} + t^{2H}
- \left|  t-s\right|  ^{2H}\right)
\]
Our focus will be on fBm with $1/3 < H < 1/2$; henceforth, we shall assume
that such an $H$ has been fixed and will drop the parameter from our notation
whenever possible to do so without causing confusion.

An $n$-dimensional fractional Brownian motion is a stochastic process
$\{B_{t} = (B^{1}_{t},\ldots,B^{n}_{t});\;t\in[0,T]\}$ is a
continuous-time process comprised of $n$ independent copies of one-dimensional
fractional Brownian motion, all having the same Hurst parameter $H$.

It is straightforward to check that the process $B$ satisfies a self-similarity property; that is to say, the processes $B_{at}$ and $a^{-H}B_t$ are equal in distribution.  
By Kolmogorov's continuity criterion, the sample paths $t \mapsto B_t$ are almost
surely H\"older continuous of order $\alpha$, for any $\alpha< H$ -- see Theorem 1.6.1 of \cite{biagini} for details of this proof.

\subsection{Dyadic Approximation}\label{dyadicapproximation}

For each $m$, we will let 

\noindent $D_{m}:=\{k2^{-m}T;k=0,1,\ldots,2^{m}\}$. We define
the $m$-th dyadic approximator $\pi_{m}:\mathcal{C}([0,T],\mathbb{R}%
^{d})\longrightarrow\mathcal{C}([0,T],\mathbb{R}^{d})$ as the unique projection operator
such that, for any given $f\in\mathcal{C}([0,T],\mathbb{R}^{d})$,
\begin{align*}
\pi_{m}f(t)  =f(t),&\quad(t\in D_{m})\\
\frac{d^{2}}{dt^{2}}\pi_{m}f(t)  =0.&\quad(t\notin D_{m})
\end{align*}
In words, $\pi_{m}f$ is nothing more than the piecewise linear path agreeing with $f$ on
the set $D_{m}$. We will regularly use the shorthand notation $f_{m}:=\pi
_{m}f$ where convenient. Similarly, we will define the $m$-th dyadic approximation of
fractional Brownian motion $B_{m}:=\pi_{m}B$; more explicitly,
\[
(B_{m})_{t}:=B_{t_{-}}+(t-t_{-})2^{m}[B_{t_{+}}-B_{t_{-}}],\quad(0\leq
t\leq T)
\]
where $t_{-}$ is the largest member of $D_{m}$ such that $t_{-}\leq t$ and
$t_{+}$ is the smallet member of $D_{m}$ such that $t\leq t_{+}.$

\subsection{$p$-variation and Rough Paths}

Let $U$ be a Banach space ($U$ will typically be $\mathbb{R}$ or
$\mathbb{R}^{2})$ with norm $\Vert\cdot\Vert_{U}$, and $\mathcal{P}[0,T]$
denote the set of finite partitions of $[0,T]$. Suppose we are given a path
$f\in\mathcal{C}([0,T],U)$; then for each $1\leq p<\infty$ and $\Pi
=\{0=t_{1},t_{2},\ldots,t_{N}=1\}\in\mathcal{P}[0,T]$, one may define the
quantities
\begin{align*}
\Delta_{i}f  &  :=f(t_{i})-f(t_{i-1}),\\
V_{p}(f:\Pi)  &  :=\left(  \sum\limits_{i=1}^{N}\Vert\Delta_{i}f\Vert_{U}%
^{p}\right)  ^{\frac{1}{p}},\\
\Vert f\Vert_{p}  &  :=\sup\limits_{\Pi\in\mathcal{P}[0,T]}V_{p}(f:\Pi).
\end{align*}
The norm $\Vert\cdot\Vert_{p}$ is referred to as the \textit{p-variation
norm}; we shall define the space $\mathcal{C}_{p}(U):=\{f\in\mathcal{C}%
([0,T],U);\;\Vert f\Vert_{p}<\infty\}$, which is a Banach space under
$\Vert\cdot\Vert_{p}$. It is easy to check that for given $\alpha$ and $p$
such that $\alpha<\frac{1}{p}$, any $\alpha$-H\"{o}lder continuous function is
in $\mathcal{C}_{p}(U)$. Also, one has that for any $1\leq p<q$,
$\mathcal{C}_{p}\left(  {U}\right)  \subset\mathcal{C}_{q}(U)$.

Given $f\in\mathcal{C}_{p}(U),g\in\mathcal{C}_{q}(U)$, where $p$ and $q$ are
such that $\frac{1}{p}+\frac{1}{q}>1$, one can develop the notion of
integration of $f$ against $g$ in the following manner: if $\{\Pi_{n}:= \{t_i\}%
\}\subset\mathcal{P}[0,T]$ is a collection of partitions such that the mesh
size $|\Pi_{n}|$ tends to 0 as $n\rightarrow\infty$, we define
\[
\int_{0}^{T}f\;dg:=\lim\limits_{n\rightarrow\infty}\sum\limits_{i=1}%
^{\#(\Pi_{n})}f(c_{i})\Delta_{i}g
\]
where $c_{i}\in(t_{i-1},t_{i})$.  This limit is guaranteed to exist under the assumptions presented, and is independent of the choice we make of the family of partitions so long as their mesh size tends to zero - see Theorem 3.3.1 of \cite{lyons_qian_2002} for further details. The element $\int_{0}^{1}f\;dg$ is referred
to as \textit{the Young's integral of f against g}.  This expression was originally formulated in \cite{young}.  We have the following
estimate on the value of this expression (see Formula 10.9 of \cite{young})
\begin{equation}\label{e.yintbound}
\left\vert \int_{0}^{T}f\;dg-\left[  f(0)\cdot\left(  g(T)-g(0)\right)
\right]  \right\vert \leq C\Vert f\Vert_{p}\Vert g\Vert_{q},
\end{equation}
where the constant $C$ depends only on the values of $p$ and $q$. 

Similarly,
given some $f\in\mathcal{C}([0,T]^{2},U)$, and partitions $\Pi_{1}%
=\{s_{i}\}$,

\noindent $\Pi_{2}=\{t_{j}\}\in\mathcal{P}[0,T]$, we may define
\begin{align*}
\Delta_{ij}f  &  :=f(s_{i},t_{j})-f(s_{i},t_{j-1})-f(s_{i-1},t_{j}%
)+f(s_{i-1},t_{j-1}),\\
V_{p}(f:\Pi_{1},\Pi_{2})  &  :=\left(  \sum\limits_{i=1}^{\#(\Pi_{1})}%
\sum\limits_{j=1}^{\#(\Pi_{2})}\Vert\Delta_{ij}f\Vert_{U}^{p}\right)  ^{\frac{1}{p}%
},\\
\Vert f\Vert_p^{(2D)}  &  :=\sup\limits_{\Pi_{1},\Pi_{2}\in\mathcal{P}%
[0,T]}V_{p}(f:\Pi_{1},\Pi_{2}).
\end{align*}
As in the (one-dimensional) case above, we shall
define the space 

\noindent $\mathcal{C}^{(2D)}_{p}(U):=\{f\in\mathcal{C}([0,T]^{2}%
,U);\;\Vert f\Vert_p^{(2D)}<\infty\}$. It is helpful to note that, for each $f\in\mathcal{C}^{(2D)}_{p}(U)$ such that $f(0,\cdot)=0$, then for each
fixed $s\in\lbrack0,T]$, $f(s,\cdot)\in\mathcal{C}_{p}(U)$ and $\|f(s,\cdot)\|_p \leq \|f\|^{(2D)}_p$, since
\begin{align*}
\sum\limits_{i=1}^{\#(\Pi)}|\Delta_{i}f(s,\cdot)|^{p}  &  =\sum_{i=1}^{\#(\Pi)}|\Delta_{i}f(s,\cdot)-\Delta_{i}f(0,\cdot)|^{p}\\
&\leq\sum_{i=1}^{\#(\Pi)}\big(|\Delta_{i}f(s,\cdot)-\Delta
_{i}f(0,\cdot)|^{p}\\
&  \qquad\quad+|\Delta_{i}f(T,\cdot)-\Delta_{i}f(s,\cdot)|^{p}\big)\leq\left(\Vert
f\Vert_p^{(2D)}\right)^{p}.
\end{align*}

Trivially, one also has that for each $f\in\mathcal{C}_p(U)$, the function

\noindent $f\otimes f: [0,T]^2 \rightarrow U$ defined by
\[
\left(f\otimes f\right) (s,t) := f(s)f(t)
\]
is contained in $\mathcal{C}_p^{(2D)}(U)$, and $\|f\otimes f\|^{(2D)}_p \leq \|f\|^2_p$.

Given $f\in\mathcal{C}^{(2D)}_{p}(U)$ and $g\in\mathcal{C}^{(2D)}_{q}(U)$, where $p$ and $q$
are such that $\frac{1}{p}+\frac{1}{q}>1$, the \textit{2D-Young's integral of
f against g}, denoted by $\int_{[0,T]^{2}}f\;dg$, is defined by
\[
\int_{\lbrack0,T]^{2}}f\;dg:=\lim\limits_{n\rightarrow\infty}\sum
\limits_{i=1}^{\#(\Pi_{n})}\sum
\limits_{j=1}^{\#(\Psi_{n})}f(c_{i},d_{j})\Delta_{ij}g,
\]
where, as before, $\{\Pi_{n} := \{t_i\} \}, \{\Psi_{n} := \{s_j\} \} \subset\mathcal{P}[0,T]$ are collections of
partitions such that the maximum mesh size $|\Pi_{n}| \vee |\Psi_n|$ tends to 0 as $n\rightarrow
\infty$ and $c_{i}\in(t_{i-1},t_{i})$, $d_j\in(s_{j-1},s_j)$.  Existence of this limit under the given assumptions, independent of the choice of the family of partitions, is proven in Theorem 1.2 of \cite{towghi}, as is an estimate similar to that
of the one-dimensional case:
\[
\left\vert \int_{\lbrack0,T]^{2}}f\;dg\right\vert \leq C\Vert g\Vert
^{(2D)}_{q}\Big(\Vert f\Vert_p^{(2D)}+\Vert f(0,\cdot)\Vert_{p}+\Vert f(\cdot
,0)\Vert_{p}+|f(0,0)|\Big).
\]

It will be helpful to record here a pair of results relating the variation of paths with their linear approximations.
\begin{theorem}[Propositions 5.20 and 5.60 of \cite{Friz-Victoir2010}]\label{t.uniformpvar}
\mbox{}
\begin{enumerate}
\item Suppose $x \in \mathcal{C}_p(U)$, and let $x_m := \pi_m x$ be the dyadic approximation to $x$ as defined above.  Then one has that
\[
\|x_m\|_p \leq 3^{p-1}\|x\|_p.
\] 
\item Suppose $x \in \mathcal{C}^{(2D)}_p(U)$, and let $x_m := \pi_m x$ be the dyadic approximation to $x$ as defined above.  Then one has that
\[
\|x_m\|^{(2D)}_p \leq 9^{p-1}\|x\|^{(2D)}_p.
\] 
\end{enumerate}
\end{theorem}
For further development of the theory, the interested reader may look in \cite{young},
\cite{friz_victoir}, and \cite{towghi}.

\subsection{Gaussian measure spaces}\label{gaussianmeasurespace}
Let $(\wiener, \norm{\cdot})$ denote a separable Banach space.  We will say that a measure $\prob$ on $\wiener$ is Gaussian if there exists a symmetric bilinear form $q: \wiener^* \times \wiener^* \longrightarrow \reals$ such that for all $\varphi \in \wiener^*$,
\[
\int_\wiener \exp\left(i\varphi(\omega)\right)\;d\prob(\omega) = \exp\left(-\half q(\varphi,\varphi)\right).
\]
Let $\borel$ refer to the Borel $\sigma$-algebra on $\wiener$; we will call the triple $(\wiener, \borel, \prob)$ a Gaussian space.
Define a continuous mapping $J: L^2(\prob) \rightarrow \wiener$ by
\[
Jf := \int_\wiener \omega f(\omega)\;d\prob(\omega).
\]
Define $\rkhs$ as the image of $J$ restricted to the space $\overline{\wiener^*}^{L^{2}(\mathbb{P})}$; this space may be equipped with inner product given by
\[
\left\langle Jf, Jg \right\rangle_\rkhs = \left\langle f,g \right\rangle_{L^2(\prob)}.
\]
We will refer to $\mathcal{H}$ as the \textit{Cameron-Martin space} associated to the Gaussian space
$(\mathcal{W},\mathcal{B},\mathbb{P})$. More information regarding the construction of these spaces may be found in \cite{bogachev}, \cite{daprato}, and \cite{kuo}.

A well-known example of a Gaussian
measure space is the one associated to Brownian motion where
\[
\mathcal{W}:=\{\omega\in\mathcal{C}([0,T],\mathbb{R}):\omega(0)=0\}
\]
and the Gaussian measure $\mathbb{P}$ on $\mathcal{W}$ is the law of standard
Brownian motion. Details of the construction of this measure may be found in \cite{wiener_1923} and \cite{wiener_1924}. 

In this case, the Cameron-Martin space is 
\[
\mathcal{H}=\mathcal{H}_{\frac{1}{2}}:=\left\{h\in\wiener: h(s) = \int_0^s \phi(u)\;du; \phi \in  L^{2}[0,T]\right\}
\]
with inner product given by
\[
\langle h,k\rangle_{\mathcal{H}_{\frac{1}{2}}}:=\int_{0}^{t}h^{\prime
}(s)k^{\prime}(s)ds.
\]

A second example of a Gaussian measure space which is pertinent to the results described below, is as follows: 
let $\wiener$ be defined as above, and define the Gaussian measure $\prob$ on $\wiener$ as the law of fractional Brownian motion with Hurst parameter $H<1/2$; then by following Proposition 2.1.2 of \cite{biagini}, we
have that $\mathcal{H}$ consists of functions of the form

\noindent $h(t)=\int_{0}^{t} K_{H}(t,s)\hat{h}(s)\;ds$, where $\hat{h}\in L^{2}[0,T]$ and
\begin{align*}
K_{H}(t,s)  &  :=b_{H}\bigg[\left(  \frac{t}{s}\right)  ^{H-\frac{1}{2}%
}\left(  t-s\right)  ^{H-\frac{1}{2}}\\
&  \qquad-\left(  H-\frac{1}{2}\right)  s^{\frac{1}{2}-H}\int_{s}%
^{t}(u-s){H-\frac{1}{2}}u^{H-\frac{3}{2}}\;du\bigg],
\end{align*}
where $b_H$ is some suitable normalization constant.  The inner product of this space is given by
\[
\langle h,k\rangle_{\mathcal{H}}:=\langle\hat{h},\hat{k}\rangle_{L^{2}[0,T]}.
\]
For each fixed $t\in\lbrack0,T]$, the function $R(t,\cdot) = \E[B_tB_\cdot] \in\mathcal{H}$ is a \textit{reproducing
kernel} for the space; that is to say, for any $h\in\mathcal{H}$, we have the
following:
\[
\langle h,R(t,\cdot)\rangle_{\mathcal{H}}=h(t).
\]

Let $\mathcal{S}$ refer to the space of \textit{cylinder functions}; that is
to say, random variables of the form
\[
F(\omega)=f(\phi_1(\omega),\ldots,\phi_n(\omega)),
\]
where $f\in C^{\infty}(\mathbb{R}^{n})$ with all partial derivatives having at
most polynomial growth, and $\{\phi_{1},\ldots,\phi_{n}\}\subset\wiener^*$. 
Let $D:\mathcal{S}\rightarrow\mathcal{S}\otimes\mathcal{H}%
^{\ast}$ be the operator defined by the action by
\begin{align*}
DF(\omega)k&:=\frac{d}{dt}|_{t=0}F\left(  \omega+tk\right)\\
&=\sum\limits_{i=1}^{n}\partial_{i}f(\phi_1(\omega),\ldots,\phi_{n}(\omega))\phi_i(k)\quad
\left(k\in\mathcal{H}\right).\\
\end{align*}
For $1\leq q<\infty,$ we will let $\mathbb{D}^{1,q}$ denote the closure of
$\mathcal{S}$ with respect to the norm
\[
\Vert F\Vert_{1,q}:=\left(  \mathbb{E}[|F|^{q}]+\mathbb{E}[\Vert
DF\Vert_{\mathcal{H}^{\ast}}^{q}]\right)
\]
One can naturally define an iterated derivate operator $D^{k}$ taking values
in $\mathcal{H}^{\otimes k}$; from this we can define the seminorm
\[
\Vert F\Vert_{k,q}:=\left(  \mathbb{E}[|F|^{q}]+\sum\limits_{j=1}%
^{k}\mathbb{E}[\Vert D^{j}F\Vert_{\mathcal{H}^{\otimes k}}^{q}]\right)
\]
and we will denote by $\mathbb{D}^{k,q}$ the closure of $\mathcal{S}$ with
respect to $\Vert\cdot\Vert_{k,q}$. Also, let
\[
\mathbb{D}^{\infty}:=\bigcap\limits_{k\in\mathbb{N}}\bigcap\limits_{q\geq
1}\mathbb{D}^{k,q}.
\]
Given some $F\in\mathbb{D}^{1,q}$, we may define the \textit{Malliavin
covariance matrix} $\gamma$ by
\[
\gamma:=DF(DF)^{\ast}.
\]

\section{Proof of Theorem \ref{mainresult}}

Again, we fix the Hurst parameter $\frac{1}{3} < H < \frac{1}{2}$.  
Let
\[
\mathcal{W}^{2}:=\mathcal{W}\otimes\mathbb{R}^{2}\cong\mathcal{W}%
\times\mathcal{W}=\{\omega\in\mathcal{C}([0,T],\mathbb{R}^{2}):\omega
(0)=\mathbf{0}\}.
\]
On $\mathcal{W}^{2}$, one may construct a unique Gaussian measure
$\mathbb{P}$ such that the coordinate process $\{B_{t}\}_{0\leq
t\leq T}$ defined by
\[
B_t(\omega) = \omega(t)
\]
is a two-dimensional fractional Brownian motion with Hurst parameter
$H$ and $\mathbb{P}=\text{Law}(B)$. Our reproducing kernel Hilbert space in this
instance is given as $\mathcal{H}^{2}=\mathcal{H}\otimes\mathbb{R}^{2}$, where
$\mathcal{H}$ is the Cameron-Martin space for one-dimensional fractional Brownian motion; for
a general element $h=(h^{1},h^{2})\in\mathcal{H}^{2}$, the norm is given by
$\Vert h\Vert_{\mathcal{H}^{2}}^{2}=\Vert h^{1}\Vert_{\mathcal{H}}^{2}+\Vert
h^{2}\Vert_{\mathcal{H}}^{2}$. 

We shall fix
\begin{equation}\label{e.pvalue}
p\in\left(\frac{1}{H},\frac{1}{1-2H}\right),
\end{equation}
and define (following Section 5.3.3 of \cite{Friz-Victoir2010}) the spaces
\begin{align*}
\wiener_p  &  :=\overline{C^{\infty}\left(
\left[  0,T\right]  ,\mathbb{R}\right)  \cap\mathcal{W}}^{\left\Vert
\cdot\right\Vert _{p}};\\
\wiener_p^2  &  :=\wiener_p\times\wiener_p\cong
\overline{C^{\infty}\left(
\left[  0,T\right]  ,\mathbb{R}^2\right)  \cap\mathcal{W}^2}^{\left\Vert
\cdot\right\Vert _{p}}.
\end{align*}
For reasons which will become apparent as we progress, it will be beneficial for us to declare that from here on out our process $\{B_t\}$ will be restricted to the probability space $(\wiener_p^2,\borel_{\wiener_p^2}, \prob|_{\wiener_p^2})$; the details of this restriction are included in the Appendix.  Most importantly, the Cameron-Martin space $\rkhs^2$ associated to the restriction of our measure is the same as the Cameron-Martin space associated to $(\wiener^2,\borel,\prob)$ as given previously.  The following proposition shows that elements of $\mathcal{H}$ live within a smaller
variational space, and will be used repeatedly in the sequel.

\begin{proposition}
\label{pvarembedding} Let $r := \frac{1}{2H}$. Then the covariance
kernel $R$ associated to fractional Brownian motion with Hurst parameter $H$
has finite two-dimensional $r$-variation. Also, the associated
reproducing kernel Hilbert space $\mathcal{H}$ may be embedded in the space
$\mathcal{C}_{r}(\mathbb{R})$; furthermore, this embedding is a contraction.
\end{proposition}
\begin{remark}
The assertions in Proposition~\ref{pvarembedding} were originally made in p. 3363 of \cite{cass_friz_victoir} or p.13 of \cite{friz_victoir_2007}.  However, some technical gaps have been found in their proof that $R \in \mathcal{C}^{(2D)}_r$; as a result, we have included a direct proof of this statement within the Appendix (see Theorem~\ref{t.Rvar} below).
\end{remark} 
Note that the above implies that $\mathcal{H}^{2}\subset\mathcal{C}_{\tilde
{p}}(\mathbb{R}^{2})$, and, since $r<p$, $\mathcal{H}%
^{2}\subset\wiener_p^2$.   Also note that
Proposition~\ref{pvarembedding} implies that $R(s,\cdot)\in\mathcal{C}%
_{r}$ for each fixed $s\in\lbrack0,T]$. 

Let $Y_{m}$ be the
solution to the differential equation
\begin{align*}
d(Y_{m})_{s}  &  =X_{1}(Y_{m})_{s}d(B_{m}^{1})_{s}+X_{2}(Y_{m})_{s}d\left(
B_{m}^{2}\right)  _{s}\\
(Y_{m})_{0}  &  =\mathbf{0}%
\end{align*}
with vector fields given by $X_{1}:=\frac{\partial}{\partial x}-\frac{1}%
{2}y\frac{\partial}{\partial z}$, $X_{2}:=\frac{\partial}{\partial y}+\frac
{1}{2}x\frac{\partial}{\partial z}$. Without a loss of generality, we may fix all of the processes in question at time $T$,
and will elect to do so for the remainder.  Explicit calculations show that
\[
\left(Y_{m}\right)_T=\left(  (B_{m})_T,\frac{1}{2}\int_{0}^{T}(B_{m}^{1})_{s}\;d(B_{m}%
^{2})_{s}-(B_{m}^{2})_{s}\;d(B_{m}^{1})_{s}\right)  .
\]




As a result of Theorem 19 and Corollary 20 of \cite{coutin_qian}, we have the existence of a process 
\[
Y_T := \lim\limits_{m\rightarrow\infty}(Y_m)_T;\quad\text{(a.s. and in}\, L^2)
\]
we may denote this process suggestively as
\[
Y_T=\left(  B_T,\frac{1}{2}\int_{0}^{T}\left(  B_{s}^{1}\;dB_{s}^{2}-B_{s}%
^{2}\;dB_{s}^{1}\right)  \right)  =\left(  B_T,\frac{1}{2}\int_{0}^{T}%
\alpha\left(  B_s,dB_s\right)  \right)
\]
(In this expression we have let $\alpha\left(
a,b\right)  :=a^{1}b^{2}-a^{2}b^{1}.)$ 

\subsection{Calculation of the Derivative of $Y_T$.}

So that we might study the properties of the processes $(Y_{m})_T$ and $Y_T$, let us
introduce the following expressions: given some $\omega:=\omega^{1}%
e_{1}+\omega^{2}e_{2}\in\mathcal{W}_{p}^{2}$, we will let $\tilde{\omega}$ be
the element of $\mathcal{W}^{2}$ defined by
\[
\tilde{\omega}:=\omega^{2}e_{1}-\omega^{1}e_{2}=J\omega,
\]
where $J$ is rotation by $-\pi/2.$ We will let $q:\mathcal{H}^{2}%
\times\mathcal{H}^{2}\rightarrow\mathbb{R}$ denote the symmetric quadratic
form given by%
\begin{align}\label{qhk}
q(h,k)&=\frac{1}{2}\left(  \int_{0}^{T}h(t)\cdot d\tilde{k}(t)+\int_{0}^{T}k(t)\cdot
d\tilde{h}(t)\right)\\\nonumber &\qquad= \int_{0}^{T} h(t)\cdot d\tilde{k}(t) + \half\left( k(T)\cdot \tilde{h}(T)\right).\\\nonumber
\end{align}
The integrals given in the definition of $q$ are to be considered as Young's
integrals; by our previous assumption that $\frac{1}{3}<H<\frac{1}{2}$, we
have that $\frac{2}{r}=4H>1$, and so $q$ is well-defined. Furthermore,
standard Young's integration bounds and Proposition~\ref{pvarembedding} gives us that
\begin{equation}\label{e.qbounds}
|q(h,k)|\leq C\Vert h\Vert_{r}\Vert k\Vert_{r}\leq
C\Vert h\Vert_{\mathcal{H}}\Vert k\Vert_{\mathcal{H}},
\end{equation}
and hence $q$ is continuous in each variable. We may now rewrite our
approximate solutions $Y_{m}$ in the following form:
\[
Y_{m}=\left(  (B_{m})_T,\frac{1}{2}q(B_{m},B_{m})\right)  .
\]
As a consequence of the representation theorem of Riesz, we have the existence
of a linear operator $Q:\mathcal{H}^{2}\rightarrow\mathcal{H}^{2}$ such that
$\langle Qh,k\rangle=q(h,k)$. 
%
\begin{remark}\label{r.epsilon}
Recall that $r = \frac{1}{2H}$; thus, our assumption on the value of $p$ as made in \eqref{e.pvalue} is sufficient to guarantee that the integrals above are well-defined, since this implies that $\frac{1}{p}+\frac{1}{r} >1$.
\end{remark}
\begin{proposition}\label{p.yiinh}
Fix $\alpha\in\mathcal{W}_{p}$; for each partition
$\Pi = \{t_i\}_{i=0}^N\in\mathcal{P}[0,T]$, define the vector $S_{\Pi}\in\mathcal{H}$ in the
following manner:
\[
S_{\Pi}(\cdot):=\sum\limits_{i=1}^{N}\alpha(c_{i})\left[
R(t_{i},\cdot)-R(t_{i-1},\cdot)\right], 
\]
where $c_{i}\in(t_{i-1},t_{i})$.  Then $\rkhs -\lim\limits_{k\rightarrow\infty}S_{\Pi_{k}}$ exists, where $\{\Pi_{k}\}_{k=1}^{\infty}\subset\mathcal{P}[0,T]$ with $|\Pi_{k}|$
converging to zero as $k \longrightarrow \infty$; furthermore, this limit is independent of the family of partitions.  We will denote this limit by
\[
\int_0^T \alpha(t)R(dt,\cdot).
\]
This limit satisfies the following properties:
\begin{enumerate}
\item $
\left\|\int_0^T \alpha(t)R(dt,\cdot)\right\|_\rkhs^2 = \int_{[0,T]^2}\alpha\otimes\alpha\;dR$;
hence, there exists a constant $C > 0$ such that
\[
\left\|\int_0^T \alpha(t)R(dt,\cdot)\right\|^2_\rkhs \leq C\|\alpha\|^2_p \|R\|^{(2D)}_r.
\]
\item For each $h \in \rkhs$, $\left\langle \int_0^T \alpha(t)R(dt,\cdot), h\right\rangle_\rkhs = \int_0^T \alpha(t)dh(t)$.
\item $\left(\int_0^T \alpha(t)R(dt,\cdot)\right)(s) = \int_0^T \alpha(t)R(dt,s)$.
\end{enumerate}
\end{proposition}
\begin{proof}
First note that $\frac{1}{p}+\frac{1}{r}>(1-2H)+2H = 1$,
which implies that 
\begin{enumerate}
\item the Young's integral of $\alpha$
against $R(\cdot,s)$ for any $s\in\lbrack0,T]$ is well-defined, and
\item the 2D-Young's integral of $\alpha \otimes \alpha$ against $R$ is well-defined.
\end{enumerate} 
It follows from Theorem 2.1 of \cite{towghi} that for each $k$,
\begin{align}
\Vert S_{\Pi_{k}}\Vert_{\mathcal{H}}^{2} &=\sum\limits_{i=1}^{\#(\Pi_k)}
\sum\limits_{j=1}^{\#(\Pi_k)}\alpha(c_{i})\alpha(c_{j})\left\langle \Delta_iR(t_{i},\cdot), \Delta_jR(t_{j},\cdot)\right\rangle_\rkhs\label{normyoungint}\\
&\leq C\|\alpha\|^2_p\|R\|^{(2D)}_r,\nonumber\\
\nonumber
\end{align}
where $C$ is a constant depending only on $p$ and $r$. 
Given 
any two partitions $\Pi_n = \{s_i\}, \Pi_m = \{t_k\}$ in the family, for $c_i \in [s_{i-1},s_i], d_k \in [t_{k-1}, t_k]$, 
\begin{align*}
\Vert S_{\Pi_{n}}&- S_{\Pi_{m}}\Vert_\rkhs^{2}=\Vert S_{\Pi_{n}}\Vert_\rkhs^{2}+\Vert
S_{\Pi_{m}}\Vert_\rkhs^{2}-2\langle S_{\Pi_{n}},S_{\Pi_{m}}\rangle_\rkhs\\
&  =\sum\limits_{i=1}^{\#(\Pi_n)}%
\sum\limits_{j=1}^{\#(\Pi_n)}\alpha(c_{i})\alpha(c_{j})\left[
\Delta_{ij}R(s_i,s_j)\right]\\
&  \qquad+\sum\limits_{k=1}^{\#(\Pi_m)}%
\sum\limits_{l=1}^{\#(\Pi_m)}\alpha(d_{k})\alpha(d_{l})\left[
\Delta_{kl}R(t_k,t_l)\right]\\
&  \qquad-\,2\sum\limits_{i=1}^{\#(\Pi_n)}%
\sum\limits_{k=1}^{\#(\Pi_m)}\alpha(c_{i})\alpha(d_{k})\left[
\Delta_{ik}R(s_i, t_k)\right]\\
&  \overset{n,m\rightarrow\infty}{\longrightarrow}\int_{[0,T]^{2}}(\alpha\otimes\alpha)(s,t)\;dR(s,t)\\
&\qquad\qquad+\int_{[0,T]^{2}}(\alpha\otimes\alpha)(s,t)\;dR(s,t)\\
&  \qquad\qquad-2\int_{[0,T]^{2}}(\alpha\otimes\alpha)(s,t)\;dR(s,t)=0.
\end{align*}
Hence, the completeness of $\rkhs$ implies the existence of $\int_0^T \alpha(t)R(dt,\cdot)$.  Since the 2D-Young's integral is independant of choice of partitions, one may also see from the calculation above that the limit of $S_{\Pi_k}$ is also independant of choice of partition, as claimed. Letting $k$ tend to infinity in \eqref{normyoungint} and applying bounds as in Theorem 1.2 of \cite{towghi} proves (1).  For an arbitrary $h \in \rkhs$, we note that
\begin{align*}
\left\langle \int_{0}^{T}\alpha(t)\;R(dt,\cdot),h\right\rangle
_{\mathcal{H}}&=\lim\limits_{|\Pi|\rightarrow0}\langle S_{\Pi},h\rangle
_{\mathcal{H}}\\
  &  =\lim\limits_{|\Pi|\rightarrow0}\sum\limits_{i=1}^{\#(\Pi)}%
\alpha(c_{i})\big\langle R(t_{i+1},\cdot)-R(t_{i},\cdot
),h\big\rangle_{\mathcal{H}}\\
&  =\lim\limits_{|\Pi|\rightarrow0}\sum\limits_{i=1}^{\#(\Pi)}\alpha%
(c_{i})\left[  h(t_{i+1})-h(t_{i})\right] \\
&  =\int_{0}^{T}\alpha(t)\;dh(t),\\
\end{align*}
and so (2) holds.  In particular, by setting $h = R(s,\cdot)$, (3) is a consequence of (2).
\end{proof}

\begin{proposition}\label{p.qact}
Let $Q:\rkhs^2 \rightarrow \rkhs^2$ be the bounded operator defined by 
\[
q(h,k) = \left\langle Qh, k \right\rangle_{\rkhs^2}.
\]
Then the action of $Q$ on elements of $\rkhs^2$ is given by
\begin{equation}\label{e.qact}
Qh :=\frac{1}{2}R\left(T,\cdot\right)  \tilde{h}\left(T\right)  -\int
_{0}^{T}\tilde{h}\left(t\right)  R\left(dt,\cdot\right)  .
\end{equation}
\end{proposition}
\begin{proof}
Pick an arbitrary $k \in \rkhs^2$.   The inner product of $k$ against each of the terms on the right hand side of \eqref{e.qact} is given as
\begin{align*}
\left\langle R\left(T,\cdot\right)  \tilde{h}\left(T\right), k\right\rangle_{\rkhs^2}  &= \tilde{h}^1(T)\left\langle R\left(T,\cdot\right), k^1\right\rangle_{\rkhs} + \tilde{h}^2(T)\left\langle R\left(T,\cdot\right), k^2\right\rangle_{\rkhs}\\
&= \tilde{h}^1(T)k^1(T) + \tilde{h}^2(T)k^2(T) = \tilde{h}(T)\cdot k(T),\\
\end{align*}
and, as a result of Proposition~\ref{p.yiinh},
\begin{align*}
\left\langle \int_{0}^{T}\tilde{h}\left(t\right)R\left(dt,\cdot\right), k\right\rangle_{\rkhs^2} &= \left\langle \int_{0}^{T}\tilde{h}^1\left(t\right)R\left(dt,\cdot\right), k^1\right\rangle_{\rkhs}\\
&\qquad + \left\langle \int_{0}^{T}\tilde{h}^2\left(t\right)R\left(dt,\cdot\right), k^2\right\rangle_{\rkhs}\\
&= \int_{0}^{T}\tilde{h}^1\left(t\right)dk^1(t) + \int_{0}^{T}\tilde{h}^2\left(t\right)dk^2(t)\\
&= \int_{0}^{T}\tilde{h}\left(t\right)\cdot dk(t) = -\int_{0}^{T}h\left(t\right)\cdot d\tilde{k}(t).\\
\end{align*}
By combining these terms and comparing to \eqref{qhk}, we see that the claim is proven.
\end{proof}
\begin{proposition}
\label{Qprops} Let $Q:\mathcal{H}^{2}\longrightarrow\mathcal{H}^{2}$ be the
operator defined above.

\begin{enumerate}

\item[1.] $Q$ may be extended to an operator from $\mathcal{W}_{p}^{2}$ into
$\mathcal{H}^{2}$, which will also be denoted by $Q$; for any $\omega
\in\mathcal{W}_{p}^{2}$,
\[
Q\omega:=\frac{1}{2}R(T,\cdot)\tilde{\omega}(T)-\int_{0}^{T}\tilde{\omega
}(t)\;R(dt,\cdot).
\]

\item[2.] Q is a bounded operator on $\wiener_p^2$.
\end{enumerate}
\end{proposition}

\begin{proof}
\begin{enumerate}

\item[1.] That $Q$ is well-defined as an operator on $\mathcal{W}_{p}^{2}$
follows from Proposition~\ref{p.yiinh}; it is then of immediate consequence that
$Q\omega\in\mathcal{H}^{2}$ for any $\omega=(\omega^{1},\omega^{2}%
)\in\mathcal{W}_{p}^{2}$.

\item[2.] 
From this fact, along with the estimate given by \eqref{normyoungint}, the
value of $\Vert Q\omega\Vert_{\mathcal{H}^{2}}^{2}$ for a fixed $\omega
\in\mathcal{W}_{p}^{2}$ may be calculated:
\begin{align*}
\Vert Q\omega\Vert_{\mathcal{H}^{2}}^{2}  &  =\sum\limits_{i=1}%
^{2}\bigg\|\frac{1}{2}R(T,\cdot)\tilde{\omega}^{i}(T)-\int_{0}^{T}%
\tilde{\omega}^{i}(t)\;R(dt,\cdot)\bigg\|_{\mathcal{H}}^{2}\\
&  =\sum\limits_{i=1}^{2}\frac{1}{4}\bigg\|R(T,\cdot)\tilde{\omega}%
^{i}(T)\bigg\|_{\mathcal{H}}^{2}+\bigg\|\int_{0}^{T}\tilde{\omega}%
^{i}(t)\;R(dt,\cdot)\bigg\|_{\mathcal{H}}^{2}\\
&  \qquad-\left\langle R(T,\cdot)\tilde{\omega}^{i}(T),\int_{0}^{T}%
\tilde{\omega}^{i}(T)\;R(dt,\cdot)\right\rangle _{\mathcal{H}}\\
&  =\sum\limits_{i=1}^{2}\frac{T^{2H}|\tilde{\omega}^{i}(T)|^{2}}%
{4}+\bigg\|\int_{0}^{T}\tilde{\omega}^{i}(t)\;R(dt,\cdot)\bigg\|_{\mathcal{H}%
}^{2}\\
&  \qquad-\tilde{\omega}^{i}(T)\int_{0}^{T}\tilde{\omega}^{i}(t)\;R(dt,T)\\
&  \leq\sum\limits_{i=1}^{2}\frac{T^{2H}|\tilde{\omega}^{i}(T)|^{2}}%
{4}+\left\vert \int_{[0,T]^{2}}\left(  \tilde{\omega}^{i}\otimes\tilde{\omega
}^{i}\right)  (s,t)\;dR(s,t)\right\vert \\
&  \qquad+\left\vert \int_{0}^{T}\tilde{\omega}^{i}(T)\tilde{\omega}^{i}%
(t)\;R(dt,T)\right\vert \\
&  \leq\sum\limits_{i=1}^{2}\Vert\tilde{\omega}^{i}\Vert_{p}%
^{2}\left(  \frac{T^{2H}}{4}+\Vert R \Vert^{(2D)}_{r} + \Vert R(T,\cdot) \Vert_{r} \right) \\
& \leq\left(  \frac{T^{2H}}{4}+2\Vert R \Vert^{(2D)}_{r}\right)\Vert\omega\Vert^2_p, \\
\end{align*}
which is finite by Proposition~\ref{pvarembedding}.
\end{enumerate}
\end{proof}

Let us denote by $QB$ the random variable taking values in $\mathcal{H}^{2}$:
\[
QB := \frac{1}{2} R(T,\cdot)\tilde{B}_{T} - \int_{0}^{T} \tilde{B}%
_{t}\;R(dt,\cdot),
\]
where $\tilde{B} := (B^2,-B^1)$.

We are now in a position to calculate the derivative of the process $Y_T$.  To begin with, for $i=1,2$ let us denote by $R_{t}^{i}$ the linear operator on $\mathcal{H}^{2}$ with
action given by
\[
R_{t}^{i}h=\langle R(t,\cdot)e_{i},h\rangle e_i=h^{i}(t)e_i,
\]
where $\{e_1,e_2\}$ is the standard basis of $\reals^2$.  It is immediate that $DB=R_{T}^{1}+R_{T}^{2}$.

\begin{proposition}
\label{p.diff}The process $Y_T$ has a derivative, $DY_T$, taking values in the space of linear operators from $\rkhs^2$ into $\reals^3$
, with action given by
\[
DY_Th=\left(  R_{T}^{1}h+R_{T}^{2}h,\langle QB
,h\rangle_{\mathcal{H}^{2}}\right)  \text{ a.s.}%
\]

\end{proposition}

\begin{proof}
Recall that $Y_T = (B_T,A_T)$, where $A_T$ was defined as the almost sure limit of processes given by \eqref{e.am}.
We have that $Y_T$ is continuously $\mathcal{H}^{2}$-differentiable by
\cite[Proposition 3]{coutin_friz_victoir}, and Corollaries 16 and 20 of
\cite{coutin_qian} imply that $\mathbb{E}\left|(Y_{m})_T-Y_T\right|^{2}\rightarrow0$ as
$m\rightarrow\infty.$ We claim that $DA_Th = \langle QB, h \rangle$; to prove this, it suffices to show that
\begin{equation}\label{e.DAconv}
\mathbb{E}\Big\|\langle QB,\cdot\rangle-D\left(\frac{1}{2}q(B_{m},B_{m})\right)
\Big\|_{(\mathcal{H}^{2})^{\ast}}^{2}\overset{m\rightarrow\infty
}{\longrightarrow}0.
\end{equation}
Recall that the process $B_{m}$ was defined as the
dyadic linear approximator to our fractional Brownian motion $B$; similarly,
we will denote by $R_{m}(u,v)$ the $m$-th dyadic approximation of the kernel
$R\in\mathcal{H}$ in the first variable; i.e.,
\[
R_{m}(u,v):=\pi_{m}(R(\cdot,v)\big)(u).
\] 
Since $B(T)\cdot\tB(T) = 0$, it follows from \eqref{qhk} that
\begin{align*}
q(B_m,B_m) &= \int_0^T (B_m)_t\cdot d(\tB_m)_t\\
&= \half\sum\limits_{k=0}^{2^{m}} \Big(B_{t_i}+B_{t_{i-1}}\Big)\cdot\left(\tB_{t_i}-\tB_{t_{i-1}}\right).\\
\end{align*}
By the product rule (see Proposition 1.2.3 of \cite{nualart}, for example),
\begin{align*}
Dq(B_m,B_m)h &= \half\sum\limits_{k=0}^{2^{m}}\bigg[\Big(h(t_i)+h(t_{i-1})\Big)\cdot\left(\tB_{t_i}-\tB_{t_{i-1}}\right)\\
&\qquad\qquad + \Big(B_{t_i}+B_{t_{i-1}}\Big)\cdot\left(\tilde{h}(t_i)-\tilde{h}(t_{i-1})\right)\bigg]\\
&= \int_0^T h_m(t)\cdot d(\tB_m)_t + \int_0^T (B_m)_t\cdot d\tilde{h}_m(t)\\
&= h(T)\cdot\tB_T - 2\int_0^T (\tB_m)_t\cdot dh_m(t).\\
& = \left\langle R(T,\cdot)\tB_T - 2\int_0^T (\tB_m)_t\;R_m(dt,\cdot), h\right\rangle_{\rkhs^2}.\\
\end{align*}
Since $\Vert\langle h,\cdot\rangle\Vert_{\mathcal{H}^{\ast}}=\Vert
h\Vert_{\mathcal{H}}$ for any Hilbert space $\mathcal{H}$, we can use the above calculations to rewrite the left side of \eqref{e.DAconv} as 
\begin{align*}
\E&\left\| QB - \left(\half R(T,\cdot)\tB_T - \int_0^T (\tB_m)_t\;R_m(dt,\cdot)\right)\right\|^2_{\rkhs^2}\\
&\qquad=\E\left\|  \int_0^T \tB_t\;R(dt,\cdot)- \int_0^T (\tB_m)_t\;R_m(dt,\cdot)\right\|^2_{\rkhs^2}.\\
\end{align*}
Hence, to prove the claim is it required for us to show that
\begin{equation}\label{e.expconv}
\mathbb{E}\left\Vert \int_{0}^{T}B^{i}_tR(dt,\cdot)-\int_{0}^{T}%
(B_{m}^{i})_t R_{m}(dt,\cdot)\right\Vert _{\mathcal{H}}^{2}\stackrel{m\rightarrow\infty}{\longrightarrow}0.
\end{equation}
We note that for any $h \in \rkhs$,
\begin{align*}
\lim\limits_{m\rightarrow\infty} &\left\langle \int_0^T (B_m)_t\;R_m(dt,\cdot), h\right\rangle_{\rkhs}\\
& = \half\Bigg[\lim\limits_{m\rightarrow\infty}\sum\limits_{k=0}^{2^{m}}\bigg((B_m)_{t_k}+(B_m)_{t_{k-1}}\bigg)\bigg(h(t_k)-h(t_{k-1})\bigg)\Bigg]\\
&= \half\Bigg[\lim\limits_{m\rightarrow\infty}\sum\limits_{k=0}^{2^{m}}(B_m)_{t_k}\bigg(h(t_k)-h(t_{k-1})\bigg)\\
&\qquad\qquad\qquad\quad + (B_m)_{t_{k-1}}\bigg(h(t_k)-h(t_{k-1})\bigg)\Bigg]\\
&= \half\left[2\int_0^T B_t\;dh(t)\right] = \left\langle \int_0^T B_t\;R(dt,\cdot),h\right\rangle_{\rkhs}\\
\end{align*}
and so $\int_0^T (B_m)_t\;R_m(dt,\cdot)$ converges pointwise in $\rkhs$ to $\int_0^T B_t \;R(dt,\cdot)$.

Applying Proposition~\ref{p.yiinh}, along with Theorem~\ref{t.uniformpvar}, we may conclude that there exists a constant $C$ for which
\begin{align*}
\left\|\int_0^T B_t\;R(dt,\cdot)\right\|^2_{\rkhs} &\leq C\|B\|^2_p\|R\|^{(2D)}_r,\\
\left\|\int_0^T (B_m)_t\;R_m(dt,\cdot)\right\|^2_{\rkhs} &\leq C\|B\|^2_p\|R\|^{(2D)}_r,\\
\end{align*}
with the second inequality being independent of $m$.  Hence, we may use Fernique's Theorem (Theorem 2.6 of \cite{daprato}) to conclude that
\begin{equation}\label{e.fernique}
\left\|\int_0^T B_t\;R(dt,\cdot) - \int_0^T (B_m)_t\;R_m(dt,\cdot)\right\|_\rkhs^2 \leq C\|B\|^2_p < \infty.
\end{equation}
Hence, we may apply Dominated Convergence to conclude that \eqref{e.expconv} holds, as desired.

\end{proof}
\begin{remark}
In fact, we have shown something slightly stronger in the above proof.  By changing the exponent on the left-hand side of \eqref{e.fernique}, we may conclude that $\left\|D(Y_m)_T - DY_T\right\|_{(\rkhs^2)^*}$, converges to zero in all $L^j, j \geq 1$.  By applying the triangle inequality, we also find that $\|DY_T\|_{(\rkhs^2)^*} \in L^{\infty -}$.
\end{remark}

\begin{proposition}\label{p.smooth}
The random variable $Y_T$ is in $\mathbb{D}^{\infty}$.
\end{proposition}

\begin{proof}

Corollaries 16 and 20 of \cite{coutin_qian} implies that $\mathbb{E}%
[|A_T|^{2}]<\infty$, and that $A_T=L^{2}-\lim\limits_{m\rightarrow\infty}%
\left(A_{m}\right)_T$. Hence, $A_T$ is in the second-order It\^{o} chaos; it follows from hypercontractivity (pp. 61-63 of \cite{nualart}, for example), that $\E[|A_T|^j] < \infty$ for all $1 \leq j < \infty$.  Combining this with the above remark, we find that
\begin{align*}
\Vert Y_T\Vert_{1,j}^{j}  &  =\mathbb{E}[|Y_T|^{j}]+\mathbb{E}\left[  \Vert
DY_T\Vert_{(\mathcal{H}^{2})^{\ast}}^{j}\right] \\
&  \leq\mathbb{E}[|B_T|^{j}]+\mathbb{E}[|A_T|^{j}] + \E\left[  \Vert
DY_T\Vert_{(\mathcal{H}^{2})^{\ast}}^{j}\right] < \infty.\\
\end{align*}
\end{proof}

\subsection{$(\det\gamma)^{-1}\in L^{\infty-}(\mathcal{W}^{2},\mathbb{P})$.}

We begin by recording some more general results, which will be useful in proving integrability of
$(\det\gamma)^{-j}$.

%
\begin{lemma}\label{l.expint}
Suppose that $X$ is a non-negative random variable such that, for each $j \geq 1$, there exists a constant $C_j > 0$ for which
\[
\E\left[e^{-sX}\right] \leq C_js^{-j}\quad\forall\, s \geq 1.
\]
Then $X^{-1} \in L^{\infty -}$.
\end{lemma}
\begin{proof}
Fix $j \geq 1$.  We note that for any $k \geq 0$,
\[
\int_0^\infty s^{j-1}e^{-ks}\;ds = k^{-j}\Gamma(j),
\]
where $\Gamma$ denotes the standard Gamma function.  By letting $k = X$, we find that
\[
\E[X^{-j}] = \frac{1}{\Gamma(j)}\E\left[\int_0^\infty s^{j-1}e^{-sX}\;ds\right] = \frac{1}{\Gamma(j)}\int_0^\infty s^{j-1}\E\left[e^{-sX}\right]\;ds.
\]
Using the assumption given, we can see that this quantity is clearly finite.
\end{proof}

\begin{theorem}
[{see Melcher \cite[pp.26-27]{melcher}}]\label{t.melcher} Let $(\wiener, \borel
,\mathbb{P})$ be a Gaussian measure space with associated Cameron-Martin space
$\mathcal{H}$, and suppose $\Phi:\wiener\times\wiener\rightarrow\mathbb{R}$ is a
bounded non-negative quadratic form. Then the operator $\hat{\Phi}:\mathcal{H}%
\rightarrow\mathcal{H}$ given by
\[
\Phi(h,k)=\langle\hat{\Phi}h,k\rangle_{\mathcal{H}}%
\]
is trace-class. In addition, if $\hat{\Phi}$ is not a finite rank
operator, then
\[
\Phi^{-1}\in L^{\infty-}(\wiener,\mathbb{P}).
\]

\end{theorem}

\begin{proof}
%
By Theorem 5.3.32 of \cite{stroock_1993}, we have that for a
set of independent, identically distributed standard normal random variables $\{\xi
_{n}\}_{n=1}^{\infty}$, the series $B^N := \sum_{n=1}^{N}\xi_{n}%
h_{n}$ converges in $\wiener$ to $B$ $\mathbb{P}$-a.s. and in all $L^j, j \geq 1$ as $N\rightarrow
\infty$, and
\[
\text{Law}\left(  \sum_{n=1}^{\infty}\xi_{n}%
h_{n}\right) = \prob.
\]
In particular, the fact that $\E\|B^N - B\|^2_\wiener \rightarrow 0$ implies that 
\[
\E|\Phi(B^N, B^N)| \rightarrow \E|\Phi(B,B)|,
\]
and Fernique's theorem allows us to conclude that
\begin{align*}
\sum\limits_{n=1}^\infty \left\langle \hat{\Phi}h_n,h_n\right\rangle &= \lim\limits_{N\rightarrow\infty} \sum\limits_{n=1}^N \Phi(h_n,h_n)\\
&= \lim\limits_{N\rightarrow\infty}  \E\left[\Phi(B^N,B^N)\right]\\
&= \E[\Phi(B)] \leq C\E[\|B\|^2_\wiener] < \infty.\\
\end{align*}
Thus, $\hat{\Phi}$ is trace-class.

Suppose that $\hat{\Phi}$ is not finite rank.  Since $\hat{\Phi}$ is compact, there exists an orthonormal basis $\{h_{n}\}_{n=1}^{\infty}\subset\mathcal{H}$ for which $\hat{\Phi}h_{n}=\lambda_{n}h_{n}$; our assumption guarantees that $\#\{n:\lambda_n > 0\} = \infty$.  Using this, it is easy to check that
\[
\Phi(B^N,B^N) = \left\langle \hat{\Phi}B^N,B^N \right\rangle_\rkhs = \sum\limits_{n=1}^N \lambda_n\xi_n^2
\]
and so
\[
\Phi(B,B) = \lim\limits_{N \rightarrow \infty} \sum_{n=1}^{N}\lambda_{n}\xi_{n}^2.
\]
We will let $K_N := \#\{1\leq n\leq N:\lambda_n > 0\}$;  it is clear that $K_N \stackrel{N\rightarrow\infty}{\longrightarrow} \infty$.  Therefore, for each fixed $N$ and positive $s$,
\begin{align*}
\mathbb{E}\left[  \exp\left(  -s\Phi\left(B,B\right)\right)  \right]   &  =\mathbb{E}\left[
\exp\left(  -s\lim\limits_{N\rightarrow\infty}\sum\limits_{n=1}^{N}%
\lambda_{n}\xi_{n}^{2}\right)  \right] \\
&  \leq\mathbb{E}\left[  \exp\left(  -s\sum\limits_{n=1}^{N}\lambda_{n}%
\xi_{n}^{2}\right)  \right]\\
&\qquad = \prod\limits_{n=1}^N \left(\frac{1}{2\lambda_n s +1}\right)^{\frac{1}{2}} \leq C_{N}s^{-\frac{K_N}{2}}.
\end{align*}
Applying Lemma~\ref{l.expint} finishes the proof.
\end{proof}

In order to apply Theorem~\ref{t.melcher}, we will explicitly calculate a formula for the determinant of the Malliavin covariance matrix associated to $Y$.

\begin{lemma}
\label{blockdet} Given any $a\neq0$, $C\in M_{m,n}(\mathbb{R})$, and $D\in
M_{n}(\mathbb{R})$, one has that
\[
\det\left[
\begin{array}
[c]{cc}%
aI_{m} & C\\
C^{tr} & D
\end{array}
\right]  =a^{m}\left(  \det(D-a^{-1}C^{tr}C)\right),
\]
where $C^{tr}$ is the transpose of $C$.
\end{lemma}

\begin{proof}
This claim follows immediately when one writes
\[
\left[
\begin{array}
[c]{cc}%
aI_{m} & C\\
C^{tr} & D
\end{array}
\right]  =\left[
\begin{array}
[c]{cc}%
aI_{m} & 0\\
C^{tr} & I_{n}%
\end{array}
\right]  \left[
\begin{array}
[c]{cc}%
I_{m} & a^{-1}C\\
0 & D-a^{-1}C^{tr}C
\end{array}
\right]  .
\]

\end{proof}
\begin{proposition}
\label{p.mal}Define the map $\gamma:\wiener_p^2 \rightarrow M^3(\reals)$ in the following manner:
\[
\gamma(\omega):=\left[
\begin{array}
[c]{cc}%
T^{2H}I_{2} &  \left(Q\omega\right)(T)\\
\left[\left(Q\omega(T)\right)\right]^{tr} & \Vert Q\omega\Vert_{\mathcal{H}^{2}%
}^{2}%
\end{array}
\right]
\]
Also, define the quadratic form $\Phi$ on $\wiener_p^2$ as follows:
\[
\Phi(\omega)=T^{4H}\Vert Q\omega\Vert_{\mathcal{H}^{2}}^{2}-T^{2H}|Q\omega(T)|^{2}.
\]
Then
\begin{enumerate}
\item $DY_T(DY_T)^* = \gamma$ a.s.
\item $\Phi = \det \gamma$.
\end{enumerate}
\end{proposition}

\begin{proof}
\begin{enumerate}

\item[1.] We begin by calculating the adjoint operator $(DY)^{\ast}$.  

Recall that $R_{t}^{i}h=h^{i}(t)e_i.$  For a fixed $\lambda=(\lambda^{1},\lambda^{2})\in\mathbb{R}^{2}$, $h \in \rkhs^2$, and $i=1,2$,
\begin{align*}
\left\langle (R_{T}^{i})^{\ast}\lambda, h \right\rangle_{\rkhs^2} &= \lambda \cdot h^{i}(T)e_i\\
&= \lambda^ih^i(T) = \left\langle \lambda^iR(T,\cdot)e_i, h\right\rangle_{\rkhs^2}.
\end{align*}
Thus, one has that
\[
R_{T}^{i}(R_{T}^{j})^{\ast}\lambda=R_{T}^{i}\lambda
^{j}R(T,\cdot)e_{j}=\delta_{ij}T^{2H}\lambda^{j}.
\]
Suppose $\mathbf{x}\in\mathbb{R}^{3}$.
\begin{align*}
\langle(DY_T)^{\ast}\mathbf{x},k\rangle &  =\mathbf{x}\cdot
DY_Tk\\
&  =\mathbf{x}\cdot(k(T),\langle QB,k\rangle)\\
&  =(x^{1},x^{2})\cdot k(T)+x^{3}(\langle QB,k\rangle)\\
&  =\langle(R_{T}^{1})^{\ast}(x^{1},x^{2})+(R_{T}^{2})^{\ast}(x^{1}%
,x^{2})+x^{3}QB,k\rangle\\
&  =\langle x^{1}R(T,\cdot)e_{1}+x^{2}R(T,\cdot)e_{2}+x^{3}QB,k\rangle.
\end{align*}
Using this, we may now verify the claim:
\begin{align*}
DY_T(DY_T)^*\mathbf{x}  &  =DY_T\left(  x^{1}R(T,\cdot)e_{1}%
+x^{2}R(T,\cdot)e_{2}+x^{3}QB\right) \\
&  =DY_T(x^{1}R(T,\cdot)e_{1})+DY_T(x^{2}R(T,\cdot)e_{2}%
)+DY_T(x^{3}QB)\\
&  =(T^{2H}x^{1},0,\langle QB,x^{1}R(T,\cdot)e_{1}\rangle)\\
&  \quad+(0,T^{2H}x^{2},\langle QB,x^{2}R(T,\cdot)e_{2}\rangle)\\
&  \quad+(x^{3}QB^{1}(T),x^{3}QB^{2}(T),\langle QB,x^{3}QB\rangle)\\
&  =\left[
\begin{array}
[c]{cc}%
T^{2H}I_{2} & QB(T)\\
\left(  QB(T)\right)  ^{tr} & \Vert QB\Vert_{\mathcal{H}^{2}}^{2}%
\end{array}
\right]  \mathbf{x}.
\end{align*}

\item[2.] This follows immediately from Lemma~\ref{blockdet}.
\end{enumerate}
\end{proof}

%

\begin{lemma}
\label{detkernel} The quadratic form $\Phi$ is positive semidefinite and has a trivial nullspace.
\end{lemma}

\begin{proof}
The Cauchy-Schwarz inequality allows us to see that 
\[
|(Q\omega)^i(T)|^2 = |\langle  R(T,\cdot),(Q\omega)^i \rangle_\rkhs|^2 \leq \|R(T,\cdot)\|_\rkhs^2\|(Q\omega)^i\|_\rkhs^2 = T^{2H}\|(Q\omega)^i\|_\rkhs^2.
\]
Hence, $\Phi$ is non-negative and $\Phi=0$ if and only if $Q\omega=\mathbf{c}%
R(T,\cdot)$ for some constant vector $\mathbf{c}\in\mathbb{R}^{2}$. By the
definition of $Q$, it would then follow that
\[
0=\int_{0}^{T}\left(  \tilde{\omega}(t)-\tilde{\omega}(T)+\mathbf{c}\right)
\;R(dt,\cdot)
\]
which implies by Proposition~\ref{p.yiinh} that
\[
0=\int_{0}^{T}\left(  \tilde{\omega}(t)-\tilde{\omega}(T)+\mathbf{c}\right)
\cdot dh(t)\quad\left(  \forall\,h\in\mathcal{H}^2\right)  .
\]
As a result of Lemma 31 of \cite{friz_victoir}, one has that $\mathcal{C}_{c}^{\infty
}(0,T)\subset\mathcal{H}$.  Hence, we can conclude that $(\tilde{\omega}(t)-\tilde
{\omega}(T))+\mathbf{c}$ is constant on $[0,T]$, which implies that $\tilde{\omega}(t)$
is constant as well. Thus $\omega(t)=\omega(0)=0$ for all $t \in [0,T]$.
\end{proof}

\begin{corollary}\label{c.int}
$(\Phi)^{-1} \in L^{\infty-}(\mathcal{W}^{2}_p,\mathbb{P})$.
\end{corollary}

\begin{proof}
Combining Lemma~\ref{detkernel} and
Theorem~\ref{t.melcher} gives us the desired result.
\end{proof}

\section{Appendix}

\subsection{Two-Dimensional $r$-Variation for Covariance of Fractional Brownian Motion}

Fix $0 < H < 1/2$ and $T > 0$.  We will, as usual, denote by $R:[0,T]^2 \longrightarrow \reals$ the covariance function for fractional Brownian motion with Hurst parameter $H$ on $[0,T]$; that is,
\[
\E[B^H_sB^H_t] = R(s,t) := \half\left[s^{2H}+t^{2H} - |t-s|^{2H}\right].
\]
To this covariance function, we may associate a finitely addivite signed measure $\mu_R$ on the algebra generated by rectangles of the form 

\noindent $\{(a,b]\times(c,d] \subset (0,T]^2\}$ such that
\begin{align*}
\mu_R((a,b]\times(c,d]) &= R(b,d)-R(a,d)-R(b,c)+R(a,c)\\
& = \half\left(|d-a|^{2H}+|c-b|^{2H}-|d-b|^{2H}-|c-a|^{2H}\right).\\
\end{align*}
It is easy to check the covariance of the process increments $B^H_b - B^H_a$ and $B^H_d - B^H_c$ is given by the $\mu_R$-measure of the rectangle $(a,b]\times(c,d]$, or
\[
\E[(B^H_d - B^H_c)(B^H_b - B^H_a)] = \mu_R((a,b]\times(c,d]).
\]
\begin{lemma}\label{l.negativecorr}
Suppose $0 \leq a < b < c < d \leq T$.  Then the process increments $B^H_d - B^H_c$ and $B^H_b - B^H_a$ are negatively correlated; i.e., $\mu_R\left((a,b]\times(c,d]\right)<0$.
\end{lemma}
\begin{proof}
We begin by noting that the following relations hold:
\begin{enumerate}
\item $(d-a)+(c-b) = (d-b)+(c-a)$.
\item $(d-a) > \left[(d-b) \vee (c-a)\right] := \max\left[(d-b),(c-a)\right]$.
\item $(c-b) < \left[(d-b) \wedge (c-a)\right] :=\min\left[(d-b),(c-a)\right]$.
\end{enumerate}
Let $f(\alpha,\beta) = \alpha^{2H}+ \beta^{2H}$.  Then for any positive constant $C$, one can check that on the region $\{(\alpha, \beta): \alpha \geq 0, \beta \geq 0, \alpha + \beta = C\}$, $f$ has a maximum at $(\frac{C}{2}, \frac{C}{2})$ and decreases as either $\alpha$ or $\beta$ are increased.  Thus, $f(d-a,c-b) < f(d-b,c-a)$, and the claim is proven.
\end{proof}
Throughout the sequel, we will make use of the fact that for $H < \half$, one has the inequalities
\begin{align*}
(x+y)^{2H} - x^{2H} &\leq y^{2H}\\
x^{\frac{1}{2H}} + y^{\frac{1}{2H}} &\leq (x+y)^{\frac{1}{2H}}\\
\end{align*}
for all $x,y\geq 0$.
\begin{lemma}\label{l.msrbound}
Let $\mu_R$ be the measure defined above.  Then for any intervals $(a,b], (c,d] \subset [0,T]$, one has the bound
\[
\left|\mu_R((a,b]\times(c,d]) \right| \leq (b-a)^{2H}\wedge (d-c)^{2H}.
\]
\end{lemma}
\begin{proof}
We will need to consider three possible cases:
\begin{itemize}
\item[(1)] One interval is nested within the other,
\item[(2)] the intervals partially overlap, or 
\item[(3)] the intervals are disjoint.
\end{itemize} 
\begin{case}{$a \leq c < d \leq b$.}\end{case}
In this scenario, the claimed upper bound is clearly $(d-c)^{2H}$.    Using this, we have that
\begin{align*}
\left|\mu_R((a,b]\times(c,d]) \right| &= \half\left|(d-a)^{2H} + (b-c)^{2H} - (b-d)^{2H} - (c-a)^{2H}\right|\\
&\leq \half\left((d-a)^{2H} - (c-a)^{2H}\right)\\
&\qquad\qquad\quad + \half\left((b-c)^{2H} - (b-d)^{2H}\right)\\
&\leq (d-c)^{2H}.
\end{align*}
\begin{case}{$a < c \leq b < d$.}\end{case}
In this case, we know that
\begin{align*}
\left|\mu_R((a,b]\times(c,d]) \right| &= \half\left|(d-a)^{2H} + (b-c)^{2H} - (d-b)^{2H} - (c-a)^{2H}\right|\\
&\leq \half\left((d-a)^{2H} - (d-b)^{2H}\right)\\
&\qquad\qquad\quad + \half\left|(b-c)^{2H} - (c-a)^{2H}\right|\\
&\leq \half(b-a)^{2H} + \half\left((b-c)^{2H} \vee (c-a)^{2H}\right)\\
&\leq (b-a)^{2H}.
\end{align*}
In a similar manner,
\begin{align*}
\left|\mu_R((a,b]\times(c,d]) \right| &= \half\left|(d-a)^{2H} + (b-c)^{2H} - (d-b)^{2H} - (c-a)^{2H}\right|\\
&\leq \half\left((d-a)^{2H} - (c-a)^{2H}\right)\\
&\qquad\qquad\quad + \half\left|(b-c)^{2H} - (d-b)^{2H}\right|\\
&\leq \half(d-c)^{2H} + \half\left((b-c)^{2H} \vee (d-b)^{2H}\right)\\
&\leq (d-c)^{2H}.
\end{align*}
\begin{case}{$a < b \leq c < d$.}\end{case}
Here, we will use the concavity inequality twice to generate the desired bound.  Firstly, we calculate that
\begin{align*}
\left|\mu_R((a,b]\times(c,d]) \right| &= \half\left|(d-a)^{2H} + (c-b)^{2H} - (d-b)^{2H} - (c-a)^{2H}\right|\\
&\leq \half\left((d-a)^{2H} - (c-a)^{2H}\right)\\
&\qquad\qquad\quad + \half\left((d-b)^{2H} - (c-b)^{2H}\right)\\
&\leq (d-c)^{2H}.
\end{align*}
In much the same manner, we find that
\begin{align*}
\left|\mu_R((a,b]\times(c,d]) \right| &= \half\left|(d-a)^{2H} + (c-b)^{2H} - (d-b)^{2H} - (c-a)^{2H}\right|\\
&\leq \half\left((d-a)^{2H} - (d-b)^{2H}\right)\\
&\qquad\qquad\quad + \half\left((c-a)^{2H} - (c-b)^{2H}\right)\\
&\leq (b-a)^{2H}.
\end{align*}
\end{proof}
\begin{theorem}\label{t.Rvar}
Let $r := \frac{1}{2H} > 1$.  Then the function $R$ has finite two-dimensional $r$-variation over $[0,T]^2$; more specifically,
\[
\|R\|^{(2D)}_{r} \leq (5T)^{2H}.
\]
\end{theorem}
\begin{proof}
Let 
\begin{align*}
\Pi &:= \{s_0 := 0 < s_1 < \ldots < s_M := T\},\\
\Psi &:= \{t_0 := 0 < t_1 < \ldots < t_N := T\}\\
\end{align*}
be two partitions of $[0,T]$.  Fix a $j \in \{1,\ldots,T\}$.  We will let $A$ be the unique integer such that $s_{A-1} \leq t_{j-1} < s_A$, and $L \geq A$ will denote the unique integer for which $s_{L-1} < t_j \leq s_L$.
\begin{figure}[ht]
\includegraphics[width=0.9\columnwidth]{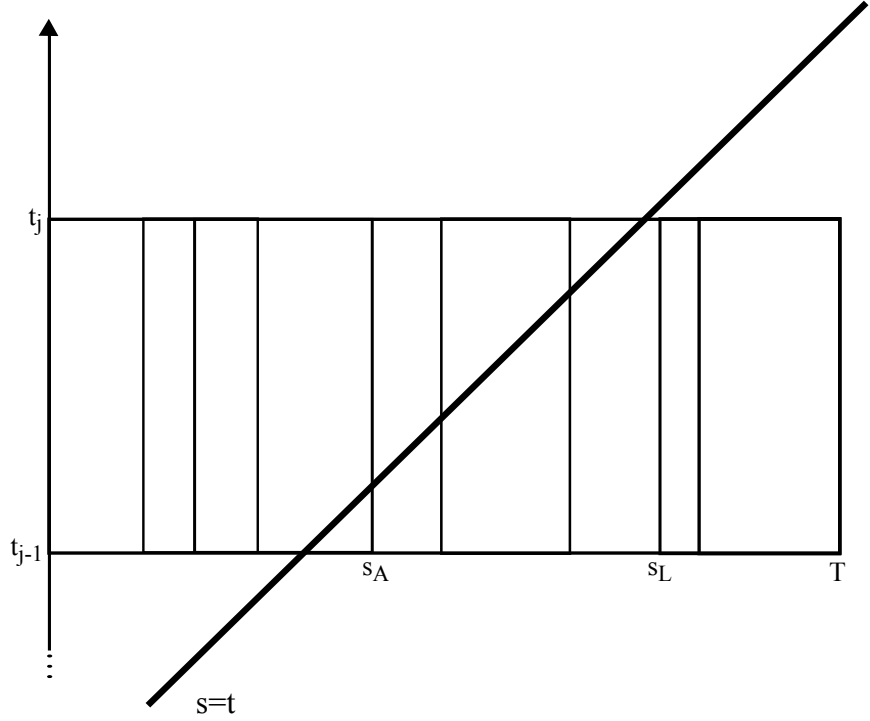}
\caption{An example partition in $s$ for a fixed strip $t_{j-1} < t \leq t_j$.}
\end{figure}
\pagebreak

As usual, we define
\[
\Delta_{ij}{R} := \mu_R\left((s_{i-1},s_i]\times(t_{j-1},t_j]\right).
\]
Then
\begin{align}\label{e.sum}
\sum\limits_{i=1}^M \left|\Delta_{ij}R\right|^{r} &\leq \sum\limits_{i=1}^{A-1} \left|\Delta_{ij}R\right|^{r} + \left|\Delta_{Aj}R\right|^{r}\\\nonumber
&\qquad + \sum\limits_{i=A+1}^{L-1} \left|\Delta_{ij}R\right|^{r} + \left|\Delta_{Lj}R\right|^{r} + \sum\limits_{i=L+1}^M \left|\Delta_{ij}R\right|^{r}.\\\nonumber
\end{align}
It follows from Lemma~\ref{l.msrbound} that
\begin{equation}\label{e.AL}
\left|\Delta_{Aj}R\right|^{r} \leq (t_j-t_{j-1}),\qquad\left|\Delta_{Lj}R\right|^{r} \leq (t_j-t_{j-1}).
\end{equation}
Lemma~\ref{l.msrbound} also implies that $\left|\Delta_{ij}R\right|^{r} \leq (s_i - s_{i-1})$; hence, we may use telescoping to bound the third term:
\begin{equation}\label{e.telescope}
\sum\limits_{i=A+1}^{L-1} \left|\Delta_{ij}R\right|^{r} \leq \sum\limits_{i=A+1}^{L-1} (s_i - s_{i-1}) = (s_{L-1} - s_A) \leq (t_j - t_{j-1}).
\end{equation}
Let us now focus on the first and last terms of Equation~\eqref{e.sum}.  Note that on each of these sums, Lemma~\ref{l.negativecorr} implies that $\Delta_{ij}R < 0$.  We may use this fact along with Lemma~\ref{l.msrbound} to see that
\begin{align}\nonumber
\sum\limits_{i=1}^{A-1} \left|\Delta_{ij}R\right|^{r} &+ \sum\limits_{i=L+1}^M \left|\Delta_{ij}R\right|^{r}\\\nonumber
&\leq 
\left(\sum\limits_{i=1}^{A-1} \left|\Delta_{ij}R\right|\right)^{r} + \left(\sum\limits_{i=L+1}^M \left|\Delta_{ij}R\right|\right)^{r}\\\nonumber
&\quad = \left|\sum\limits_{i=1}^{A-1} \Delta_{ij}R\right|^{r} + \left|\sum\limits_{i=L+1}^M \Delta_{ij}R\right|^{r}\\\nonumber
&\quad = \left|\mu_R\left((0,s_{A-1}]\times(t_{j-1},t_j]\right)\right|^{r}\\\nonumber
 &\qquad\quad+ \left|\mu_R\left((s_{L},T]\times(t_{j-1},t_j]\right)\right|^{r}\\\label{e.ends}
&\quad\leq 2(t_j - t_{j-1}).\\\nonumber
\end{align}
Combining Equations~\eqref{e.sum}--\eqref{e.ends} allows us to conclude that
\[
\sum\limits_{i=1}^M \left|\Delta_{ij}R\right|^{r} \leq 5(t_j - t_{j-1}).\\
\]
Hence,
\[
\sum\limits_{j=1}^N \sum\limits_{i=1}^M \left|\Delta_{ij}R\right|^{r} \leq \sum\limits_{j=1}^N 5(t_j - t_{j-1}) = 5T.
\]
This completes the proof, since the two-dimensional $r$-variation of $R$ is given as
\[
\|R\|^{(2D)}_{r} =  \left(\sup\limits_{\Pi,\Psi \in \mathcal{P}[0,T]}\sum\limits_\Pi\sum\limits_\Psi\left|\Delta_{ij}R\right|^{r}\right)^{\frac{1}{r}} \leq (5T)^{2H}.
\]
\end{proof}

\subsection{Restriction of Gaussian Measures}
At first blush, it may seem natural to have our process $B$ have the classical Wiener space $\wiener^2 := \mathcal{C}([0,T],\reals^2)$ as its sample space.  However, doing so is not ideal, since many of the operators we will be considering are only defined on smaller spaces, such as the $p$-variation spaces.  

We begin with a general result regarding $\sigma$-algebras.
\begin{lemma}\label{sigmasubalg}
\label{l.eq}Let $X$ be any real separable Banach space and $\mathcal{L}$ be any
Then $\left\Vert \cdot
\right\Vert _{X}$ is $\sigma\left(  \mathcal{L}\right)  $ -- measurable if and only if $\mathcal{B}_{X}=\sigma\left(  \mathcal{L}\right)  .$
\end{lemma}

\begin{proof}
It it easy to see that, in any case, $\sigma(\mathcal{L}%
)\subset\mathcal{B}_X.$  Also, since $\left\Vert \cdot\right\Vert _{X}$ is
continuous it is always Borel measurable; therefore, if $\mathcal{B}_{X}%
=\sigma\left(  \mathcal{L}\right)  $ then $\left\Vert \cdot\right\Vert _{X}$
is clearly $\sigma\left(  \mathcal{L}\right)  $ -- measurable.

Suppose that $\left\Vert \cdot\right\Vert _{X}$ is $\sigma\left(
\mathcal{L}\right)  $ -- measurable; then for each $x_{0}\in\sigma\left(  \mathcal{L}\right)$,  $\left\Vert \cdot
-x_{0}\right\Vert_X $ is also $\sigma\left(
\mathcal{L}\right)  $ -- measurable, and
$x\rightarrow x-x_{0}$ is $\sigma\left(  \mathcal{L}\right)  /\sigma\left(
\mathcal{L}\right)  $ -- measurable. From this observation, it follows that
$\sigma\left(  \mathcal{L}\right)  $ contains all balls in $X.$ Since $X$ is
separable, every open subset of $X$ may be written as a countable union of open
balls.  It follows, then, that $\sigma\left(  \mathcal{L}\right)  $
contains all open subsets of $X$ and therefore that $\mathcal{B}_{X}\subset
\sigma(\mathcal{L}).$
\end{proof}

\begin{theorem}\label{gaussianrestriction}
Suppose $(X,\mathcal{B} = \mathcal{B}_X,\mu)$ is a Gaussian probability space, and $\tilde{X}$ is a linear subspace of $X$.  Also let $\|\cdot\|_{\tilde{X}}$ is a norm on $\tilde{X}$ such that
\begin{enumerate}
\item The space $(\tilde{X},\|\cdot\|_{\tilde{X}})$ is a separable Banach space,
\item The embedding of $\tilde{X}$ into $X$ is continuous,
\item $\tilde{X} \in \mathcal{B}$ and $\mu(\tilde{X}) = 1$,
\item $\tilde{\mathcal{B}} := \mathcal{B}_{\tilde{X}} = \{A \cap \tilde{X}:A \in \mathcal{B}\}$.
\end{enumerate}
Then $\tilde{\mu} := \mu|_{\tilde{X}}$ is a Gaussian measure and $(\tilde{X},\tilde{\mathcal{B}},\tilde{\mu})$ is a Gaussian probability space.
Furthermore, $(X,\mu)$ and $(\tX,\tmu)$ share the same Cameron-Martin space $\rkhs$.
\end{theorem}

\begin{proof}
Let $R_{\pi /4}:X \times X \rightarrow X\times X$ is the rotation map defined by
\[
R_{\pi/4}(x,y) = \left(\frac{\sqrt{2}}{2}(x-y),\frac{\sqrt{2}}{2}(x+y)\right);
\]
then by the rotational invariance of Gaussian measures (see, for example, Theorem 3.1.1 of \cite{bryc}), proving the statement that $\tilde{\mu}$ is Gaussian is equivalent to proving that
\[
\int_{\tilde{X}\times{\tilde{X}}} f(x,y)\;d\tmu(x)d\tmu(y) = \int_{\tX \times \tX}f\circ R_{\pi/4} (x,y)\;d\tmu(x)d\tmu(y)
\]
for any bounded $\tilde{\mathcal{B}}\times\tilde{\mathcal{B}}$-measurable function $f$.  Let $f$ be such a function; since $\tX$ is of full $\mu$ measure, we may extend $f$ to an $\borel \times \borel$-measurable function (which we shall also refer to as $f$) such that $\int_{\tX \times \tX} f d\mu d\mu = \int_{X \times X} f d\mu d\mu$ (this extension may be done by setting a function equal to $f$ on $\tX \times \tX$ and equal to zero on the complement, for example).  Then it follows that
\begin{align*}
\int_{\tilde{X}\times{\tilde{X}}} f(x,y)\;d\tmu(x)d\tmu(y) &= \int_{\tilde{X}\times{\tilde{X}}} f(x,y)\;d\mu(x)d\mu(y)\\
					&= \int_{X\times X} f(x,y)\;d\mu(x)d\mu(y)\\
					&= \int_{X\times X} f\circ R_{\pi/4}(x,y)\;d\mu(x)d\mu(y)\\
					&= \int_{\tX\times \tX} f\circ R_{\pi/4}(x,y)\;d\mu(x)d\mu(y)\\
					&= \int_{\tX\times \tX} f\circ R_{\pi/4}(x,y)\;d\tmu(x)d\tmu(y).\\
\end{align*}
This proves the first assertion.

To see the equivalence of Cameron-Martin spaces, we recall that $J: L^2(X,\mu) \rightarrow X$, defined by
\[
Jf := \int_X xf(x)\;d\mu(x),
\]
maps onto $\rkhs$.  Again, by virtue of $\mu$ being fully supported on $\tX$, we may extend any element of $L^2(\tX,\tmu)$ to an element of $L^2(X,\mu)$; thus it is easy to see that $J(L^2(\tX,\tmu)) = J(L^2(X,\mu)) = \rkhs$, as desired.
\end{proof}
\begin{remark} An alternate proof of the equivalence of Cameron-Martin spaces may be found in Proposition 2.8 of \cite{daprato}.
\end{remark}

Let us now focus on restricting the law of fractional Brownian motion with Hurst parameter $1/3 < H < 1/2$ to a variational space.  The standard Gaussian space on which fBm is realized is $(\wiener, \borel, \prob)$, where $\wiener = \{\omega \in \mathcal{C}([0,T],\reals): \omega(0) = 0\}$ and $\prob = \text{Law}(B^H)$.  Pick $0 < \epsilon << 1$ and fix $p:= 1/H + \epsilon$.  
Let $\phi_t, 0\leq t \leq T$ denote the evaluation map on $\wiener$; i.e., $\phi_t(x) = x(t)$ for any $x \in \wiener$.  
Since
\[
\|\cdot\|_\wiener = \sup\limits_{0\leq t \leq T} \phi_t,
\]
it follows that $\|\cdot\|_\wiener$ is a $\sigma(\{\phi_t:0\leq t \leq T\})$-measurable function, and by Lemma~\ref{sigmasubalg}, it then follows that $\sigma(\{\phi_t:0\leq t \leq T\}) = \borel_\wiener$.  Recall that we have defined the $p$-variation norm on $\wiener$ by
\[
\|x\|_p = \sup\limits_{\Pi\in\mathcal{P}[0,T]} \left(\sum\limits_{i=1}^{(\#\Pi)}|\Delta_{i}x|%
^{p}\right)^{\frac{1}{p}}.
\]
Recall that we have defined the space
\begin{align*}
\wiener_p &= \overline{\{x \in \mathcal{C}_\infty([0,T],\reals):x(0)=0\}}^{\|\cdot\|_p}.\\
\end{align*}
By Corollary 5.35 and Proposition 5.38 of \cite{friz_victoir}, this space is a separable Banach space under the $p$-variation norm and contains all $q$-variation paths starting at zero for any $1 \leq q < p$.  
Note that for $x \in \wiener_p$, H\"{o}lder's inequality implies that for any $t \in [0,T]$,
\begin{align*}
|x(t)| &= |x(t) - x(0)|\\
					&\leq |x(t) - x(0)|+|x(T)-x(0)|\\
					&\leq 2^{\frac{p-1}{p}}\left(|x(t) - x(0)|^p+|x(T)-x(0)|^p\right)^{\frac{1}{p}}\\
					&\leq 2^{\frac{p-1}{p}}\|x\|_p,\\
\end{align*}
from which it follows that $\|x\|_\wiener \leq \|x\|_p$, and so the embedding of $\wiener_0^p$ into $\wiener$ is continuous.	
Observe that we may rewrite the $p$-variation norm as
\[
\|\cdot\|_p = \sup\limits_{\Pi\in\mathcal{P}[0,T]} \left(\sum\limits_{i=1}^{(\#\Pi)}|\phi_{t_i} - \phi_{t_{i-1}}|%
^{p}\right)^{\frac{1}{p}}.
\]
Thus, $\|\cdot\|_p$ is $\sigma(\{\phi_t|_{\wiener_p}: 0 \leq t \leq T\})$-measurable, which implies that $\sigma(\mathcal{L}) = \borel_{\wiener_p}$.
Furthermore, by Theorem 5.33 of \cite{friz_victoir}, we know that the space $\wiener_p$ is equivalent to
\[
\left\{x\in\wiener_p: \lim\limits_{\delta\rightarrow 0}\sup\limits_{\Pi\in\mathcal{P}[0,T]:|\Pi|<\delta}\sum\limits_{i=1}^{\#(\Pi)}|x(t_i)-x(t_{i-1})|^p = 0 \right\}
\]
If we now define
\[
\alpha_p(x) := \lim\limits_{n\rightarrow \infty}\sup\limits_{\Pi\in\mathcal{P}[0,T]\cap\mathbb{N}:|\Pi|<\frac{1}{n}}\sum\limits_{i=1}^{\#(\Pi)}|x(t_i)-x(t_{i-1})|^p,
\]
then it follows that $\alpha_p$ is a $\sigma(\{\phi_t|_{\wiener_p}:0\leq t \leq T\})$-measurable function, and that  
\[
\wiener_p = \wiener_p \cap \{\alpha_p = 0\} \in \borel_\wiener.
\]
Additionally, we may now use Lemma~\ref{sigmasubalg} to conclude that
\begin{align*}
\borel_{\wiener_p} &= \sigma(\{\phi_t|_{\wiener_p}:0\leq t \leq T\})\\
											& = \{A\cap \wiener_p:A \in \sigma(\{\phi_t|_{\wiener}:0\leq t \leq T\}) \}\\
											& = \{A\cap \wiener_p:A \in \borel_\wiener \}.\\
\end{align*}
Finally, we note that since the paths $t \mapsto B_t^H$ are a.s. H\"older continuous of order $\beta := H\left(1+\frac{\epsilon H}{2}\right)^{-1} < H$, each such path has finite $q$-variation for $q = \frac{1}{\beta} = \frac{1}{H} + \frac{\epsilon}{2}$.  So by Corollary 5.35 of \cite{friz_victoir}, $\prob(\wiener_p) \geq \prob(\wiener_q) = 1$.  Thus, we may appeal to Theorem~\ref{gaussianrestriction} to conclude that $(\wiener_p, \borel_{\wiener_p}, \prob|_{\wiener_p})$ is also a Gaussian probability space, and that the associated Cameron-Martin space $\rkhs$ coincides with the usual Cameron-Martin space corresponding to $\prob$ on $\wiener$.

\bibliographystyle{plain}
\bibliography{references}

\begin{thebibliography}{10}

\bibitem{baudoin_hairer}
Fabrice Baudoin and Martin Hairer.
\newblock A version of {H}\"ormander's theorem for the fractional {B}rownian
  motion.
\newblock {\em Probab. Theory Related Fields}, 139(3-4):373--395, 2007.

\bibitem{biagini}
Francesca Biagini, Yaozhong Hu, Bernt {\O}ksendal, and Tusheng Zhang.
\newblock {\em Stochastic calculus for fractional {B}rownian motion and
  applications}.
\newblock Probability and its Applications (New York). Springer-Verlag London
  Ltd., London, 2008.

\bibitem{bogachev}
Vladimir~I. Bogachev.
\newblock {\em Gaussian measures}, volume~62 of {\em Mathematical Surveys and
  Monographs}.
\newblock American Mathematical Society, Providence, RI, 1998.

\bibitem{bryc}
W{\l}odzimierz Bryc.
\newblock {\em The normal distribution}, volume 100 of {\em Lecture Notes in
  Statistics}.
\newblock Springer-Verlag, New York, 1995.
\newblock Characterizations with applications.

\bibitem{cass_friz_victoir}
Thomas Cass, Peter Friz, and Nicolas Victoir.
\newblock Non-degeneracy of {W}iener functionals arising from rough
  differential equations.
\newblock {\em Trans. Amer. Math. Soc.}, 361(6):3359--3371, 2009.

\bibitem{coutin_friz_victoir}
Laure Coutin, Peter Friz, and Nicolas Victoir.
\newblock Good rough path sequences and applications to anticipating stochastic
  calculus.
\newblock {\em Ann. Probab.}, 35(3):1172--1193, 2007.

\bibitem{coutin_qian}
Laure Coutin and Zhongmin Qian.
\newblock Stochastic analysis, rough path analysis and fractional {B}rownian
  motions.
\newblock {\em Probab. Theory Related Fields}, 122(1):108--140, 2002.

\bibitem{daprato}
Giuseppe Da~Prato and Jerzy Zabczyk.
\newblock {\em Stochastic equations in infinite dimensions}, volume~44 of {\em
  Encyclopedia of Mathematics and its Applications}.
\newblock Cambridge University Press, Cambridge, 1992.

\bibitem{friz_victoir}
Peter Friz and Nicolas Victoir.
\newblock A note on the notion of geometric rough paths.
\newblock {\em Probab. Theory Related Fields}, 136(3):395--416, 2006.

\bibitem{friz_victoir_2007}
Peter Friz and Nicolas Victoir.
\newblock Differential {E}quations {D}riven by {G}aussian {S}ignals {I}.
\newblock 2007.
\newblock ar{X}iv: 0707.0313.

\bibitem{Friz-Victoir2010}
Peter~K. Friz and Nicolas~B. Victoir.
\newblock {\em Multidimensional stochastic processes as rough paths}, volume
  120 of {\em Cambridge Studies in Advanced Mathematics}.
\newblock Cambridge University Press, Cambridge, 2010.
\newblock Theory and applications.

\bibitem{kuo}
Hui~Hsiung Kuo.
\newblock {\em Gaussian measures in {B}anach spaces}.
\newblock Lecture Notes in Mathematics, Vol. 463. Springer-Verlag, Berlin,
  1975.

\bibitem{lyons_qian_2002}
Terry Lyons and Zhongmin Qian.
\newblock {\em System control and rough paths}.
\newblock Oxford Mathematical Monographs. Oxford University Press, Oxford,
  2002.
\newblock Oxford Science Publications.

\bibitem{lyons}
Terry~J. Lyons.
\newblock Differential equations driven by rough signals.
\newblock {\em Rev. Mat. Iberoamericana}, 14(2):215--310, 1998.

\bibitem{malliavin}
Paul Malliavin.
\newblock {\em Stochastic analysis}, volume 313 of {\em Grundlehren der
  Mathematischen Wissenschaften [Fundamental Principles of Mathematical
  Sciences]}.
\newblock Springer-Verlag, Berlin, 1997.

\bibitem{melcher}
Tai Melcher.
\newblock {\em Hypoelliptic heat kernel inequalities on {L}ie groups}.
\newblock PhD thesis, University of California, San Diego,
  http://www.math.ucsd.edu/\~{}bdriver/DRIVER/Graduate\_Students/melcher/thesi%
s.pdf, 2004.

\bibitem{norros}
I.~Norros.
\newblock On the use of fractional brownian motion in the theory of
  connectionless networks.
\newblock {\em Selected Areas in Communications, IEEE Journal on}, 13(6):953
  --962, aug. 1995.

\bibitem{nualart}
David Nualart.
\newblock {\em The {M}alliavin calculus and related topics}.
\newblock Probability and its Applications (New York). Springer-Verlag, Berlin,
  second edition, 2006.

\bibitem{nualart_saussereau}
David Nualart and Bruno Saussereau.
\newblock Malliavin calculus for stochastic differential equations driven by a
  fractional {B}rownian motion.
\newblock {\em Stochastic Process. Appl.}, 119(2):391--409, 2009.

\bibitem{rogerswilliams2}
L.~C.~G. Rogers and David Williams.
\newblock {\em Diffusions, {M}arkov processes, and martingales. {V}ol. 2}.
\newblock Cambridge Mathematical Library. Cambridge University Press,
  Cambridge, 2000.
\newblock It{\^o} calculus, Reprint of the second (1994) edition.

\bibitem{shiryaev}
Albert~N. Shiryaev.
\newblock {\em Essentials of stochastic finance}, volume~3 of {\em Advanced
  Series on Statistical Science \& Applied Probability}.
\newblock World Scientific Publishing Co. Inc., River Edge, NJ, 1999.
\newblock Facts, models, theory, Translated from the Russian manuscript by N.
  Kruzhilin.

\bibitem{stroock_1993}
Daniel~W. Stroock.
\newblock {\em Probability theory, an analytic view}.
\newblock Cambridge University Press, Cambridge, 1993.

\bibitem{towghi}
Nasser Towghi.
\newblock Multidimensional extension of {L}. {C}. {Y}oung's inequality.
\newblock {\em JIPAM. J. Inequal. Pure Appl. Math.}, 3(2):Article 22, 13 pp.
  (electronic), 2002.

\bibitem{wiener_1923}
N.~Wiener.
\newblock Differential space.
\newblock {\em Journal of Mathematics and Physics}, 2:131--174, 1923.

\bibitem{wiener_1924}
N.~Wiener.
\newblock Un probl\`eme de probabilit\'e d\'enombrables.
\newblock {\em Bull. Soc. Math. France}, 52:569--578, 1924.

\bibitem{young}
L.~C. Young.
\newblock An inequality of the {H}\"older type, connected with {S}tieltjes
  integration.
\newblock {\em Acta Math.}, 67(1):251--282, 1936.

\end{thebibliography}

\end{document}